\documentclass[10pt,a4paper]{article}
\usepackage[a4paper]{geometry}
\usepackage{amssymb,latexsym,amsmath,amsfonts,amsthm}
\usepackage{graphicx}

\usepackage{epsfig}

\usepackage{epsfig}

\usepackage{tikz}
\usepackage{amssymb}
\usepackage{mathrsfs}
\usepackage{aliascnt}
\usepackage[colorlinks,linkcolor=blue,citecolor=blue]{hyperref}

\usepackage{pst-node}
\usepackage{tikz-cd}
\usepackage{mathtools}
\newcommand{\supp}{\mathrm{supp}}

\newtheorem{theorem}{Theorem}[section]
\newtheorem{lemma}[theorem]{Lemma}

\newtheorem{corollary}[theorem]{Corollary}

\newtheorem{HP}{Hermite-Pad\'e approximation  problem}

\theoremstyle{definition}
\newtheorem{definition}[theorem]{Definition}

\theoremstyle{remark}

\numberwithin{equation}{section}

\usepackage{scalerel,stackengine}
\stackMath
\newcommand\reallywidehat[1]{%
\savestack{\tmpbox}{\stretchto{%
  \scaleto{%
    \scalerel*[\widthof{\ensuremath{#1}}]{\kern-.6pt\bigwedge\kern-.6pt}%
    {\rule[-\textheight/2]{1ex}{\textheight}}
  }{\textheight}%
}{0.5ex}}%
\stackon[1pt]{#1}{\tmpbox}%
}
\title{Mixed type Hermite-Pad\'e approximation inspired by the Degasperis-Procesi equation}

\date{\today}

\author{G. L\'opez Lagomasino\footnotemark[1], S. Medina Peralta\footnotemark[2],  and J. Szmigielski\footnotemark[3]}

\begin{document}

\maketitle
\renewcommand{\thefootnote}{\fnsymbol{footnote}}
\footnotetext[1]{Departamento de
Matem\'aticas, Universidad Carlos III de Madrid, Avda. Universidad
30, 28911 Legan\'es, Madrid, Spain. email: lago\symbol{'100}math.uc3m.es. This author was supported
by research grant  MTM2015-65888-C4-2-P of Ministerio de Econom\'{\i}a, Industria y Competitividad.}
\footnotetext[2]{Instituto de Matem\'atica y F\'{\i}isica. Universidad de Talca.  Camino Lircay S/N, Campus Norte,  Talca, Chile. This author received support from research grant Conicyt Fondecyt/Postdoctorado/ Proyecto 3170112}.
\footnotetext[3]{Department of Mathematics and Statistics, University of Saskatchewan, Saskatoon, Canada. Partially supported by NSERC.}

\begin{abstract}
In this work we present new results on the convergence of diagonal sequences of certain mixed type Hermite-Pad\'e approximants of a Nikishin system.
The study is motivated by a
mixed Hermite-Pad\'e approximation scheme used in the construction of solutions of a Degasperis-Procesi peakon problem and germane to the analysis of the inverse spectral problem for the discrete cubic string.
\end{abstract}
\medskip


\medskip


\maketitle

\section{Introduction}
\label{section:intro}
This article deals with the convergence of sequences of a  mixed type  Hermite-Pad\'e approximation. Hermite-Pad\'e approximation was introduced by Ch. Hermite \cite{Her} for proving the trascendence of the number $e$ and subsequently it has been used in other number theory related problems (for a survey of such applications see \cite{walter}). In recent years, they have received increasing attention because of their applicability in other areas such as  non-intersecting
brownian motions theory\cite{Daems}, the study of multiple orthogonal polynomials ensembles \cite{Kuij}, random matrix theory \cite{Ble,Kuij2}, and in the solution of the Degasperis-Procesi (DP) differential equation (see, for example, \cite{2,3,jacek}). This paper is motivated in an approximation problem relevant to  the solution of the DP equation.

\medskip

\subsection{The Degasperis-Procesi equation and an approximation problem.}
\noindent

In \cite{jacek}, the authors study the following partial differential equation
\begin{equation}\label{PDE}
u_t-u_{xxt}+(b+1)uu_{x}=bu_xu_{xx}+uu_{xxx}, \quad \quad (x,t)\in \mathbb{R}^{2}.
\end{equation}
It is known that this equation  is  completely integrable if and only if $b = 2$ or $b = 3$. The case $b = 2$
is the well-known Camassa-Holm (CH) shallow water equation \cite{4}. The case $b = 3$ is
the Degasperis-Procesi (DP) equation, found by Degasperis and Procesi \cite{7}, and  subsequently  shown by
Degasperis, Holm, and Hone \cite{5, 6} to be integrable.
All equations in the family (\ref{PDE}) admit (in a weak sense) a type of non smooth
solutions called multipeakons (peakon = peaked soliton). These take the form of a train
of peak-shaped interacting waves,
\begin{equation}\label{peakons}
u(x, t) =\sum_{i=1}^{n}m_i(t)e^{-|x-x_i(t)|}.
\end{equation}
This \textit{Ansatz} is then substituted into \eqref{PDE}, resulting
in a system of ordinary differential equations on unknown smooth functions
$x_i(t)$ and $m_i(t)$.

\medskip

To solve this system the authors of \cite{jacek} consider a certain boundary value problem, called the
discrete cubic string. This problem is the main tool to obtain the explicit formulas, but it is also an interesting problem in its own right from the point of view of operator theory.  By the
forward cubic string problem we mean the following third-order spectral problem: for a given positive measure $g(y)$,
determine the eigenvalues $z$  for which nontrivial  continuous eigenfunctions
$\phi(y)$ satisfy
\begin{equation}\label{CSP}
\phi_{yyy}(y) = zg(y) \phi(y), \quad \quad \text{ for } y\in (-1, 1), \quad \quad \phi(-1) = \phi_{y}(-1) = 0, \quad \phi(1) = 0,
\end{equation}
 in a distributional sense.  If the
singular support of $g(y)$ contains the endpoints then the
values of $\phi$ and its derivatives at $-1,1$ are replaced with
the left hand, right hand limits respectively.

\medskip

This spectral problem is proved in  \cite{jacek} to be equivalent under a change of variables to the one appearing
in the DP Lax pair.  It can be viewed as a non-self-adjoint generalization of the well-known (self-adjoint)
string equation
\begin{equation}\label{SP}
\phi_{yy}= zg(y)\phi(y) \quad \quad \text{ for } y\in (-1, 1), \quad \quad
\phi(-1) = 0,\quad  \phi(1) = 0,
\end{equation}
studied by M.G. Krein in the 50's \cite{krein-string}.

\medskip

The discrete case arises when $g(y)$ is a discrete measure; in other words $g =\sum_{i=1}^{N}g_i\delta_{y_i}$.  Since
the point masses $g_i$ are placed at positions $y_i$ and there are no
masses between the points the eigenfunctions
are piecewise linear in $y$ for the ordinary string, and piecewise quadratic for the cubic
string.

\medskip

The discrete (ordinary) string plays a crucial role in finding the general $n$-peakon
solution for the CH equation \cite{2,3}. The inverse spectral problem consists in determining
the positions $y_i$ and masses $g_i$ given the eigenfrequencies and suitable additional information
about the eigenfunctions (encoded in the spectral measure of the string or,
equivalently, in its Cauchy transform). The solution presented in \cite{2} relies on the work of
T. Stieltjes \cite{Sti}, as well as its interpretation by
M.G. Krein \cite{krein-string} as  a special case of the inverse string problem; it involves
Stieltjes continued fractions, the classical moment problem, Pad\'e approximation, and orthogonal polynomials.

\medskip

The remarkable fact is that in both cases (CH and DP) the associated spectral problems have a finite positive
spectrum; this is not so surprising in the case of the ordinary string which is a self-adjoint
problem, but it is quite unexpected for the cubic string, since the problem is non-self-adjoint and
there is no \textit{a priori} reason for the spectrum to even be real, much less positive.

\medskip

Though the inverse cubic string problem is not the main concern of this paper, in \autoref{inverse} we will show how its solution is connected with an approximation problem which we will present shortly using a terminology and notation  more convenient for our purpose.

\medskip

Given two measures $\sigma_1, \sigma_2$ whose supports are contained on the real line and have at most one point in common, suppose that the following functions are well defined in the complement of the support of $\sigma_1$
\[\widehat{s}_{1,1}(z) = \int \frac{d\sigma_1(x)}{z-x}, \qquad \widehat{s}_{1,2}(z) = \int  \frac{d\sigma_1(x)}{z-x} \int \frac{d\sigma_2(x)}{x-y}\]
The pair $(\widehat{s}_{1,1},\widehat{s}_{1,2})$ constitutes what is known as a Nikishin system of functions (of order 2). Interchanging the roles of $\sigma_1,\sigma_2$ we can define in the same manner the Nikishin system $(\widehat{s}_{2,2},\widehat{s}_{2,1})$.

\medskip

\begin{HP}\label{Nik2}
Consider the systems $(\widehat{s}_{1,1},\widehat{s}_{1,2})$ and  $(\widehat{s}_{2,2},\widehat{s}_{2,1})$. Then for each $n \in \mathbb{N},$ we seek polynomials $(a_{n,0},a_{n,1}, a_{n,2}),$ not all identically equal to zero, with $\deg a_{n,0}\leq n-1$, $\deg a_{n,1}\leq n-1,$ and $\deg a_{n,2}\leq n$,  that satisfy:
\begin{align}
\left(a_{n,0}-a_{n,1}\widehat{s}_{1,1}+a_{n,2}\widehat{s}_{1,2}\right)(z)=\mathcal{O}(1/z^{n+1}) \label{JLS1},\\
\left(a_{n,1}-a_{n,2}\widehat{s}_{2,2}\right)(z)=\mathcal{O}(1/z). \label{JLS2}
\end{align}
\end{HP}

\medskip

In the inverse cubic string problem, the measures $\sigma_1,\sigma_2$ are connected with the Weyl functions of the spectral problem \eqref{CSP}. In the situation considered in \cite{2,3} the measures are discrete. With the degree of generalization presented here this problem was proposed in \cite{Bertola:CBOPs}.

\medskip

In the present paper, in Theorem \ref{equiv} we study the existence and uniqueness of the solution of an analogous approximation problem as well as the location of the zeros of the Nikishin polynomials for systems of order $m\geq 2$, biorthogonality properties satisfied by the polynomials $a_{n,m}$ are given in Theorem \ref{TBIO}, and the limit behavior of the Nikishin polynomials is described in Theorem \ref{CTPm}.

\subsection{Nikishin systems.} \label{subsec:NS} In Hermite-Pad\'e approximation the object of approximation is a system of analytic functions. We restrict our attention to so called Nikishin systems which contain, in particular, the functions appearing in equations (\ref{JLS1}) and (\ref{JLS2}). Nikishin systems were first introduced in \cite{Nik}. We will use a more general definition given in \cite{FL4} which is more appropriate for our purpose.

\medskip

In the sequel $\Delta$ denotes an interval contained in the real axis. By $\mathcal{M}(\Delta)$ we denote the class of all Borel measures $s$ with constant sign whose support  consists on infinitely many points and is    contained in $\Delta$ such that  $x^\nu \in L_1(s)$ for all $\nu \in \mathbb{Z}_+$. We denote the Cauchy transform of $s$ by
\[ \displaystyle \widehat{s}(z) = \int\frac{d s(x)}{z-x}.
\]
We have
\begin{equation}\label{expansion}
\displaystyle \widehat{s}(z)  \sim  \sum_{j=0}^{\infty} \frac{c_j}{z^{j+1}}, \qquad c_j = \int x^j ds(x).
\end{equation}
If  the support  of $s$, $\mbox{supp}(s)$, is bounded  the series is convergent in a neighborhood of $\infty$; otherwise, the expansion is asymptotic at $\infty$. That is, for each $k \geq 0$
\[ \lim_{z \to \infty} z^{k+1}\left(\widehat{s}(z) - \sum_{j=0}^{k-1} \frac{c_j}{z^{j+1}}\right) = c_k,
\]
where the limit is taken along any curve which is non tangential to $\mbox{supp}(s)$ at $\infty$.

\medskip

Now, let $\Delta_{\alpha}, \Delta_{\beta}$ be two intervals contained in the real line with at most one common point. Take $\sigma_{\alpha} \in {\mathcal{M}}(\Delta_{\alpha})$ and $\sigma_{\beta} \in {\mathcal{M}}(\Delta_{\beta})$ such that $\widehat{\sigma}_{\beta} \in L_1(\sigma_{\alpha})$. Then, using the  differential notation, we define a third measure $\langle \sigma_{\alpha},\sigma_{\beta} \rangle$ as follows
\[d \langle \sigma_{\alpha},\sigma_{\beta} \rangle (x) := \widehat{\sigma}_{\beta}(x) d\sigma_{\alpha}(x).\]
In consecutive products of measures such as $\langle \sigma_{\gamma},  \sigma_{\alpha},\sigma_{\beta} \rangle :=\langle \sigma_{\gamma}, \langle \sigma_{\alpha},\sigma_{\beta} \rangle \rangle,$ we assume not only that $\widehat{\sigma}_{\beta} \in L_1(\sigma_{\alpha})$ but also $\langle \sigma_{\alpha},\sigma_{\beta} \widehat {\rangle} \in L_1(\sigma_{\gamma})$, where $\langle \sigma_{\alpha},\sigma_{\beta} \widehat{\rangle}$ denotes the Cauchy transform of $\langle \sigma_{\alpha},\sigma_{\beta}  {\rangle}$.

\medskip

Consider a collection  $\Delta_j, j=1,\ldots,m,$ of intervals such that
\[ \Delta_j \cap \Delta_{j+1} = \emptyset, \qquad \mbox{or} \quad \Delta_j \cap \Delta_{j+1} = \{x_{j,j+1}\}, \quad j=1,\ldots,m-1,
\]
where $x_{j,j+1}$ is a single point. Let $(\sigma_1,\ldots,\sigma_m)$ be a system of measures such that $\mbox{\rm Co}(\mbox{\rm supp} (\sigma_j)) = \Delta_j, \sigma_j \in {\mathcal{M}}(\Delta_j), j=1,\ldots,m,$  where $\mbox{\rm Co}(E)$ denotes the convex hull of the set $E$. Denote
\begin{equation*}
\langle \sigma_{j},\ldots,\sigma_k  {\rangle} := \langle \sigma_j,\langle \sigma_{j+1},\ldots,\sigma_k\rangle\rangle\in {\mathcal{M}}(\Delta_j),  \qquad  1 \leq j < k\leq m.
\end{equation*}
If $\Delta_j \cap \Delta_{j+1} = \{x_{j,j+1}\}$ we also assume that $x_{j,j+1}$ is not a mass point of either $\sigma_j$ or $\sigma_{j+1}$.

\begin{definition} With the notation above, we say that ${\bf s}=(s_{1,1},\ldots,s_{1,m}) = {\mathcal{N}}(\sigma_1,\ldots,\sigma_m)$, where
\begin{equation} \label{eq:ss}
s_{1,1} = \sigma_1, \quad s_{1,2} = \langle \sigma_1,\sigma_2 \rangle, \ldots \quad , \quad s_{1,m} = \langle \sigma_1, \sigma_2,\ldots,\sigma_m  \rangle.
\end{equation}
is the \textit{Nikishin system} of measures generated by $(\sigma_1,\ldots,\sigma_m)$. The corresponding Nikishin system of functions will be denoted  by ${\bf \widehat{s}}=\left(\widehat{s}_{1,1},\ldots,\widehat{s}_{1,m}\right)$, where $\widehat{s}_{1,j}$ is the Cauchy transform of $s_{1,j}$.
\end{definition}

This definition extends the one given in  \cite{Nik} by allowing the generating measures to have unbounded support and/or have consecutive $\Delta_j$ with a common endpoint. That the generating measures have infinite support is not required for the definition of Nikishin systems; however, this condition is frequently used in the proof of the main results. If a measure has discrete support its Cauchy transform reduces to a rational function and the arguments used in the proof of some results must be modified. Sometimes the statement of the results themselves become obvious. For example, if that is the case in equations (\ref{JLS1}) and (\ref{JLS2}), the left hand sides of those relations become identically equal to zero, for all $n$ larger than the number of mass points, and the question about convergence of the approximants is trivial.

\medskip

In what follows, for $1\leq j\leq k\leq m$, we denote
\begin{equation} \label{eq:sjk}
s_{j,k} := \langle \sigma_j,\sigma_{j+1},\ldots,\sigma_k \rangle, \qquad s_{k,j} := \langle \sigma_k,\sigma_{k-1},\ldots,\sigma_j \rangle.
\end{equation}
In particular, with the collection of measures $(\sigma_1,\ldots,\sigma_m)$, we can also define the reversed Nikishin system $(s_{m,m},\ldots,s_{m,1}) = \mathcal{N}(\sigma_m,\ldots,\sigma_1)$.

\subsection{Statement of main results.}

Equations (\ref{JLS1}) and (\ref{JLS2}) suggests the following extensions to the case of Nikishin systems with $m\geq 2$ measures.

\begin{definition}{[Direct/reversed Hermite-Pad\'e approximation]} \label{MTPm1}

\noindent Consider the Nikishin systems $\mathcal{N}(\sigma_1, \sigma_2,\ldots, \sigma_{m})$ and  $\mathcal{N}(\sigma_m, \sigma_{m-1},\ldots, \sigma_1)$. Then, for each $n \in \mathbb{N},$ there exist polynomials $a_{n,0},a_{n,1},\ldots, a_{n,m}$, with $\deg a_{n,j}\leq n-1, j=0,1\ldots,m-1$,  and $\deg a_{n,m}\leq n$, not all identically equal to zero, called \textit{direct/reversed (DR) Hermite-Pad\'e polynomials} that satisfy:
\begin{align}
\left(a_{n,0}-a_{n,1}\widehat{s}_{1,1}+a_{n,2}\widehat{s}_{1,2}\cdots+ (-1)^{m}a_{n,m}\widehat{s}_{1,m}\right)(z)=\mathcal{O}(1/z^{n+1})\label{tipoIm}\\
(-1)\left(a_{n,1}-a_{n,m}\widehat{s}_{m,2}\right)(z)=\mathcal{O}(1/z)\label{tipoIIcm}\\
..........................................................\nonumber\\
(-1)^{m-2}\left(a_{n,m-2}-a_{n,m}\widehat{s}_{m,m-1}\right)(z)=\mathcal{O}(1/z)\label{tipoIIbm}\\
(-1)^{m-1}\left(a_{n,m-1}-a_{n,m}\widehat{s}_{m,m}\right)(z)=\mathcal{O}(1/z)\label{tipoIIam}.
\end{align}
\end{definition}

Alternatively, we could extend the approximation problem as follows.

\begin{definition}{[Multi-level Hermite-Pad\'e approximation]}\label{MTPm2}

\noindent Consider the Nikishin system $\mathcal{N}(\sigma_1, \sigma_2,\ldots, \sigma_{m})$. Then, for each $n \in \mathbb{N},$ there exist  polynomials $a_{n,0},a_{n,1},\ldots, a_{n,m}$ with $\deg a_{n,j}\leq n-1, j=0,1\ldots,m-1,$  and $\deg a_{n,m}\leq n$, not all identically equal to zero, called \textit{multi-level (ML) Hermite-Pad\'e polynomials} that verify:
\begin{align}
\mathcal{A}_{n,0}(z) := \left(a_{n,0}-a_{n,1}\widehat{s}_{1,1}+a_{n,2}\widehat{s}_{1,2}\cdots+ (-1)^{m}a_{n,m}\widehat{s}_{1,m}\right)(z)=\mathcal{O}(1/z^{n+1})\label{tipoIam}\\
\mathcal{A}_{n,1}(z) :=\left(-a_{n,1}+a_{n,2}\widehat{s}_{2,2}-a_{n,3}\widehat{s}_{2,3}\cdots+ (-1)^{m}a_{n,m}\widehat{s}_{2,m}\right)(z)=\mathcal{O}(1/z)\label{tipoIbm}\\
........................................................................................\nonumber\\
\mathcal{A}_{n,m-1}(z) :=\left((-1)^{m-1}a_{n,m-1}+(-1)^{m}a_{n,m}\widehat{s}_{m,m}\right)(z)=\mathcal{O}(1/z)\label{tipoIdm}.
\end{align}
\end{definition}

Notice that in this formulation the reversed Nikishin system $\mathcal{N}(\sigma_m, \sigma_{m-1},\ldots, \sigma_1)$ does not appear explicitly. On the other hand, the interpolation conditions involve all Nikishin systems at the ``inner'' levels; that is,
$\mathcal{N}(\sigma_1, \sigma_2,\ldots, \sigma_{m})$, $\mathcal{N}(\sigma_2, \sigma_3,\ldots, \sigma_{m})$, $\ldots$, $(s_{m,m})=\mathcal{N}(\sigma_m)$.  However, in Section 3 we prove

\begin{theorem}\label{equiv}
For each fixed $n$, the DR Hermite-Pad\'e polynomials and the ML Hermite-Pad\'e polynomials coincide and
the vector polynomial $(a_{n,0},a_{n,1},\ldots, a_{n,m})$ is uniquely determined except for  constant multiples.
Additionally, $\deg a_{n,j}=n-1, j=0,\ldots,m-1,$  and $\deg a_{n,m}=n$. Moreover, the zeros of $a_{n,m-1}$ and $a_{n,m}$ are all simple and lie in $\stackrel{\circ}\Delta_m$ (the interior of $\Delta_m$  with the Euclidean topology of $\mathbb{R}$).
\end{theorem}

In both cases, finding the polynomials $(a_{n,0},a_{n,1},\dots,a_{n,m})$ reduces to solving a homogeneous linear system of $n(m+1)$ equations (the interpolations conditions) on $n(m+1)+1$ unknowns (the coefficients of the polynomials); therefore, the corresponding system of equations has a non-trivial solution.

\medskip

For a fixed $n \in \mathbb{N},$ consider the vector $(b_{n,0},\ldots,b_{n,m})$ of ML Hermite-Pad\'e polynomials associated with the reversed Nikishin system $\mathcal{N}(\sigma_m,\ldots,\sigma_1)$. That is, the vector is non null and
	\begin{align}
		\left(b_{n,0}-b_{n,1}\widehat{s}_{m,m}+b_{n,2}\widehat{s}_{m,m-1}\cdots+ (-1)^{m}b_{n,m}\widehat{s}_{m,1}\right)(z)=\mathcal{O}(1/z^{n+1})\label{dtipoIam}\\
		\left(-b_{n,1}+b_{n,2}\widehat{s}_{m-1,m-1} \cdots+ (-1)^{m}b_{n,m}\widehat{s}_{m-1,1}\right)(z)=\mathcal{O}(1/z)\label{dtipoIbm}\\
		........................................................................................\nonumber\\
		\left((-1)^{m-1}b_{n,m-1}+(-1)^{m}b_{n,m}\widehat{s}_{1,1}\right)(z)=\mathcal{O}(1/z)\label{dtipoIdm}.
	\end{align}

Let
\[
		K(x_1,x_2)=\frac{1}{x_1-x_2}
\]
denote the usual Cauchy kernel. For $m> 2$	we define the Cauchy convolution kernel in the following manner
\[
		K(x_1,x_m)=\int_{\Delta_2}\int_{\Delta_3}\cdots \int_{\Delta_{m-1} }
		\frac{\, d \sigma_{m-1}(x_{m-1})\cdots \, d\sigma_3(x_3)d\sigma_2(x_2)   }
		{(x_{m-1}-x_m)(x_{m-2}-x_{m-1})\cdots\,(x_2-x_3)(x_1-x_2)  }.
\]

\begin{theorem}\label{TBIO}
The sequences of polynomials $\left(a_{n,m}(z)\right)_{n\in \mathbb{N}}$ and $\left(b_{k,m}(z)\right)_{k\in \mathbb{N}}$ are biorthogonal with respect to the Cauchy convolution kernel and the measures $(\sigma_1,\sigma_2,\ldots, \sigma_m)$; that is,
\begin{equation}\label{Bi}
	\int_{\Delta_1}\int_{\Delta_m} b_{k,m}(x_1)K(x_1,x_m) a_{n,m}(x_m)  d\sigma_m(x_m)d\sigma_1(x_1) =h_n \delta_{n,k}, \qquad h_n\neq 0.
	\end{equation}
where $\delta_{k,n} =0, k\neq n,$ and $\delta_{n,n} = 1$.
\end{theorem}

This type of biorthogonality has been discussed previously, among others,  in \cite{Bertola-Bothner,Bertola:CBOPs}.

\medskip

We are mainly concerned with the convergence properties of the sequence of vector rational functions
\[ \left(\frac{a_{n,0}}{a_{n,m}},\ldots,\frac{a_{n,m-1}}{a_{n,m}}\right), \qquad n\in \mathbb{N}.\]
Taking into consideration  the interpolation conditions and the relation
\begin{equation}
\label{rel-fund}
 0\equiv  \widehat{s}_{m,1}  + \sum_{j=1}^{m-1}(-1)^{j}   \widehat{s}_{m,j+1} \widehat{s}_{1,j} +
  (-1)^m\widehat{s}_{1, m}, \qquad z \in \mathbb{C} \setminus (\Delta_{1} \cup \Delta_m),
\end{equation}
whose proof may be found in \cite[Lemma 2.9]{FL4} (and is not difficult to verify), one can expect that under appropriate assumptions the limit should be the system of functions $(\widehat{s}_{m,1},\ldots,\widehat{s}_{m,m})$. This prediction is consistent with the convergence properties of type II and type I Hermite-Pad\'e approximants studied in \cite[Theorem 1]{Bus} (see also \cite{LF3}, \cite{GRS}, \cite{Stahl}) and \cite[Theorem 1.4]{LS}, respectively. For the definition of type I and type II Hermite-Pad\'e approximation see Subsection \ref{tipoI-II}.

\medskip

Let $\Delta \subset \mathbb{R}$ and $\sigma \in \mathcal{M}(\Delta)$. We say that $\sigma$ satisfies Carleman's condition \cite{Car} if
\begin{equation}\label{Carle} \sum_{\nu \geq 0} |c_\nu|^{-1/2\nu} = \infty,
\end{equation}
where $c_\nu = \int x^\nu d\sigma(x)$ denotes the $\nu$-th moment of $\sigma$.

\medskip

For a measure $\sigma$ supported on an interval of the form $[a,+\infty)$ or $(-\infty,a], a\in \mathbb{R},$ Carleman's condition guarantees that the corresponding moment problem is determinate. Stieltjes' theorem \cite{Sti} states that if the moment problem for $\sigma$ is determinate  then the diagonal sequence of Pad\'e approximants of $\widehat{\sigma}$ converges. If $\supp(\sigma)$ is bounded the moment problem is determinate and Markov's theorem follows (see \cite{Mar}).

\begin{theorem}\label{CTPm}
For each $n \in \mathbb{N}$, let $a_{n,0},a_{n,1},\ldots, a_{n,m}$ be the collection of
ML (or DR) Hermite-Pad\'e polynomials associated with the Nikishin system $\mathcal{N}(\sigma_1, \sigma_2,\ldots, \sigma_{m})$.
Suppose that either the sequence of moments of $\sigma_m$ satisfies Carleman's condition or  $\Delta_{m-1}$ is a bounded interval contained in $\mathbb{C} \setminus \Delta_{m}$. Then, for $j=0,\ldots m-1$
\begin{equation}\label{Con01m}
\lim_{n\rightarrow \infty}  \frac{a_{ n,j}}{a_{ n,m}} =  \widehat{s}_{m,j+1},
\end{equation}
uniformly on each compact subset $\mathcal{K} \subset \mathbb{C} \setminus \Delta_m$. Moreover
\begin{equation}\label{Con00m}
\lim_{n\rightarrow \infty}  (-1)^{j}\frac{a_{n,j}(z)}{a_{n,m}(z)}+\sum_{k=j+1}^{m-1}(-1)^{k}\frac{a_{n,k}(z)}{a_{n,m}(z)}\widehat{s}_{j+1,k}(z)+(-1)^m\widehat{s}_{j+1,m}(z)=0
\end{equation}
uniformly on each compact subset $\mathcal{K} \subset \mathbb{C} \setminus (\Delta_{j+1} \cup \Delta_m)$.
\end{theorem}

The limit of the sequence  $(a_{n,0}/a_{n,m},\ldots, a_{n,m-1}/a_{n,m}), n\in \mathbb{N},$ is the same as for type I Hermite-Pad\'e approximation of $\mathcal{N}(\sigma_1, \sigma_2,\ldots, \sigma_{m})$ and type II Hermite-Pad\'e approximation of $\mathcal{N}(\sigma_m, \sigma_{m-1},\ldots, \sigma_{1})$. For details, see \cite[Theorem 1.4]{LS} and \cite[Theorem 1]{Bus}.

\medskip

In \cite{jacek}, the authors study the case where the solutions (\ref{peakons}) are formed by a  finite linear combination of
single peakon terms $m_i e^{-|x-x_i|}$. In that case, the Cauchy transform of the spectral measures $(s_{1,1},s_{1,2},s_{2,2},s_{2,1})$ are rational fractions and for $n$ sufficiently large we have  exact equalities in (\ref{JLS1}) and (\ref{JLS2}). If we are interested in the case of  peakon solutions (\ref{peakons})  formed by an infinite number of peakons, or if we study the cubic string problem for which the weight $g(y)$ is not a discrete measure we need to deal with the convergence of the corresponding mixed type Hermite-Pad\'e approximants. Thus Theorem \ref{CTPm} opens a new direction of research aimed at the
construction of general solutions to the DP equation using peakon approximations.

\section{Proof of the main results.}

\subsection{Type I and type II Hermite-Pad\'e polynomials.} \label{tipoI-II}

We are considering a combination of type I and type II Hermite-Pad\'e polynomials which have received considerable attention for its many applications. Our construction falls in the category of mixed type Hermite-Pad\'e approximation.

\medskip

Consider a Nikishin system $(s_{1,1},\ldots,s_{1,m}) = \mathcal{N}(\sigma_1,\ldots,\sigma_m)$.
Fix ${\bf n}=(n_1,\ldots,n_m)\in {\mathbb{Z}}_+^m \setminus \{{\bf 0}\}$, ${\mathbb{Z}}_+=\{0,1,2,\ldots\}.$  Their exist polynomials $Q_{\bf n},P_{{\bf n},1},\ldots P_{{\bf n},m} $, called type II Hermite-Pad\'e polynomials of ${\bf \widehat{s}} = (\widehat{s}_{1,1},\ldots,\widehat{s}_{1,m})$ with respect to  ${\bf n}$ that satisfy:
\begin{itemize}
\item[i)]  $\deg Q_{\bf n} \leq n_1+\cdots+n_m$, $Q_{\bf n}\not \equiv 0$,
\item[ii)] $
(Q_{\bf n}  \widehat{s}_{1,j} - P_{{\bf n},j})(z)=\mathcal{O}(1/z^{n_j+1}),  \qquad j=1,\ldots,m.$
\end{itemize}
On the other hand, the collection of polynomials $a_{{\bf n},0},\ldots,a_{{\bf n},m}$ is called a type I Hermite-Pad\'e polynomial of ${\bf \widehat{s}}$ with respect to $\bf n$ if:
\begin{itemize}
\item[iii)] $\deg a_{{\bf n},j} \leq n_j -1, j=1,\ldots,m,$ not all identically equal to zero,
\item[iv)] $(a_{{\bf n},0} + \sum_{j=1}^m (-1)^j a_{{\bf n},j}\widehat{s}_{1,j})(z) = \mathcal{O}(1/z^{n_1+\cdots+n_m}).$
\end{itemize}

When $m=1$ both definitions coincide with  classical diagonal Pad\'e approximation. In contrast with Pad\'e approximation, when $m \geq 2$ the uniqueness (up to constant multiples) of these polynomials is not a trivial matter and was solved  positively in \cite{FL3,FL4} for Nikishin systems. However, for arbitrary systems of functions uniqueness is not true in general.

\medskip

Notice that the ML or DR Hermite Pad\'e polynomials combine interpolation conditions of the form (ii) and (iv) and are therefore called of mixed type.

\subsection{Some auxiliary results and concepts.}

The following lemma will be used in the proof of Theorems \ref{equiv} and \ref{CTPm}. Let us define the linear forms with polynomial coefficients
\begin{equation} \label{eq:Lj}
\mathcal{L}_j := \ell_j + \sum_{k=j+1}^{m} \ell_{k} \widehat{s}_{j+1,k}, \qquad j=0,\ldots,m-1, \qquad \mathcal{L}_m = \ell_m,
\end{equation}
where the $\ell_j$ are arbitrary polynomials.

\begin{lemma} \label{lem:2} Let $(s_{1,1},\ldots,s_{1,m}) = \mathcal{N}(\sigma_1,\ldots,\sigma_m)$ be given. Then, for each $j=0,\ldots,m-2$ and $r=j+1,\ldots,m-1$
\begin{equation}
\label{eq:aux}
\mathcal{L}_j + \sum_{k=j+1}^{r} (-1)^{k-j}\widehat{s}_{k,j+1} \mathcal{L}_{k} = \ell_j + (-1)^{r-j}\sum_{k= r+1}^{m} \ell_{k} \langle s_{r+1,k}, s_{r,j+1}\widehat{\rangle}. \end{equation}
\end{lemma}

\begin{proof}
Fix $j, \, 0\leq j \leq m-2,$ and let $r=j+1$. Then the left hand side of  \eqref{eq:aux} equals
\[\mathcal{L}_j - \widehat{s}_{j+1,j+1} \mathcal{L}_{j+1} = \ell_j + \sum_{k=j+2}^m \ell_{k} \left(\widehat{s}_{j+1,k} - \widehat{s}_{j+1,j+1}\widehat{s}_{j+2,k}\right).\]
Formula \eqref{rel-fund} applied to the Nikishin system of two measures $\mathcal{N}(s_{j+1,j+1},s_{j+2,k})$ gives
\[\langle {s}_{j+1,j+1},s_{j+2,k} \widehat{\rangle} - \widehat{s}_{j+1,j+1}\widehat{s}_{j+2,k} + \langle s_{j+2,k},{s}_{j+1,j+1} \widehat{\rangle} \equiv 0.\]
However, $s_{j+1,k} = \langle s_{j+1,j+1},s_{j+2,k}\rangle$, hence
\[\mathcal{L}_j - \widehat{s}_{j+1,j+1} \mathcal{L}_{j+1} = \ell_j - \sum_{k=j+2}^m \ell_{k} \langle s_{j+2,k},{s}_{j+1,j+1} \widehat{\rangle}\]
as needed. For $j=m-2$ the proof is complete.

\medskip

Now, fix $j < m-2$ and suppose that \eqref{eq:aux} is true for some $r,\, j+1 \leq r \leq m-2,$ and let us prove that it also holds for $r+1$. Using the induction hypothesis, we obtain
\[\mathcal{L}_j + \sum_{k=j+1}^{r+1} (-1)^{k-j}\widehat{s}_{k,j+1} \mathcal{L}_{k} =  \mathcal{L}_j + \sum_{k=j+1}^{r} (-1)^{k-j}\widehat{s}_{k,j+1} \mathcal{L}_{k} + (-1)^{r+1-j} \widehat{s}_{r+1,j+1} \mathcal{L}_{r+1}=  \]
\[ \ell_j + (-1)^{r-j}\sum_{k= r+1}^{m} \ell_{k} \langle s_{r+1,k}, s_{r,j+1}\widehat{\rangle} + (-1)^{r+1-j} \widehat{s}_{r+1,j+1} \mathcal{L}_{r+1} = \]
\[\ell_j + (-1)^{r-j}\ell_{r+1}  \widehat{s}_{r+1,j+1} + (-1)^{r-j}\sum_{k= r+2}^{m} \ell_{k} \langle s_{r+1,k}, s_{r,j+1}\widehat{\rangle} + (-1)^{r+1-j} \widehat{s}_{r+1,j+1} \mathcal{L}_{r+1} = \]
\[ \ell_j +  (-1)^{r-j}\sum_{k= r+2}^{m} \ell_{k} \left(\langle s_{r+1,j+1}, s_{r+2,k}\widehat{\rangle} - \widehat{s}_{r+1,j+1} \widehat{s}_{r+2,k}\right)  =\]
\[\ell_j +  (-1)^{r+1-j}\sum_{k= r+2}^{m} \ell_{k} \langle s_{r+2,j+1}, s_{r+1,j+1} \widehat{\rangle},\]
as claimed. In the second last step, we use  the identity
\[\langle s_{r+1,k}, s_{r,j+1} {\rangle} = \langle \langle s_{r+1,r+1}, s_{r+2,k} \rangle, s_{r,j+1} {\rangle} = \langle \langle s_{r+1,r+1}, s_{r,j+1}  \rangle, s_{r+2,k} {\rangle} = \langle s_{r+1,j+1}, s_{r+2,k} {\rangle}, \]
while in the last one we use
\[\langle s_{r+1,j+1}, s_{r+2,j+1}\widehat{\rangle} - \widehat{s}_{r+1,j+1} \widehat{s}_{r+2,k} + \langle s_{r+2,j+1}, s_{r+1,j+1} \widehat{\rangle} \equiv 0, \]
which is formula \eqref{rel-fund} applied to the Nikishin system of two measures $\mathcal{N}({s}_{r+1,j+1},s_{r+2,k})$.
\end{proof}

\medskip

We will make frequent use of \cite[Theorem 1.3]{LS}. For convenience of the reader, we state it here as a lemma.

\begin{lemma} \label{reduc} Let $(s_{1,1},\ldots,s_{1,m}) = \mathcal{N}(\sigma_1,\ldots,\sigma_m)$ be given. Assume that there exist polynomials with real coefficients $\ell_0,\ldots,\ell_m$ and a polynomial $w$ with real coefficients whose zeros lie in $\mathbb{C} \setminus \Delta_1$  such that
\[\frac{\mathcal{L}_0(z)}{w(z)} \in \mathcal{H}(\mathbb{C} \setminus \Delta_1)\qquad \mbox{and} \qquad \frac{\mathcal{L}_0(z)}{w(z)} = \mathcal{O}\left(\frac{1}{z^N}\right), \quad z \to \infty,
\]
where $\mathcal{L}_0  := \ell_0 + \sum_{k=1}^m \ell_k  \widehat{s}_{1,k} $ and $N \geq 1$. Let $\mathcal{L}_1  := \ell_1 + \sum_{k=2}^m \ell_k  \widehat{s}_{2,k} $. Then
\begin{equation} \label{eq:3}
\frac{\mathcal{L}_0(z)}{w(z)} = \int \frac{\mathcal{L}_1(x)}{(z-x)} \frac{d\sigma_1(x)}{w(x)}.
\end{equation}
If $N \geq 2$, we also have
\begin{equation} \label{eq:4}
\int x^{\nu}  \mathcal{L}_1(x)  \frac{d\sigma_1(x)}{w(x)} = 0, \qquad \nu = 0,\ldots, N -2.
\end{equation}
In particular, $\mathcal{L}_1$ has at least $N -1$ sign changes in  $\stackrel{\circ}{\Delta}_1 $.
\end{lemma}

Let us advance the following partial result.

\begin{lemma} \label{degree}
For each fixed $n\in \mathbb{N}$, the ML Hermite-Pad\'e polynomial $a_{n,m}$ has degree $n$ or it is identically equal to zero.
\end{lemma}

\begin{proof} Fix $n\in \mathbb{N}$ and let $(a_{n,0},\ldots,a_{n,m})$ be the corresponding $ML$ Hermite-Pad\'e polynomials. Everywhere below $\mathcal{A}_{n,j}, j=0,\ldots,m-1$ are the forms in Definition \ref{MTPm2} and $\mathcal{A}_{n,m} = a_{n,m}$. Let us show that for each   $j=0,\ldots,m-1,$  there exists a polynomial $w_{n,j}$ with real coefficients whose zeros lie in $\mathbb{C} \setminus \Delta_{j+1}$ such that
\begin{align}\label{partida}
\frac{\mathcal{A}_{n,j}(z)}{w_{n,j}(z)}=\mathcal{O}(1/z^{n+1}), \qquad \frac{\mathcal{A}_{n,j}(z)}{w_{n,j}(z)}\in \mathcal{H}(\mathbb{C} \setminus \Delta_{j+1}).
\end{align}

Due to \eqref{tipoIam}, if $j=0$ one can take $w_{n,0}\equiv 1$. Let us assume that the statement is true for some $j \in \{0,\ldots,m-2\}$ and let us show that it also holds for $j+1$.

\medskip

Using  \eqref{eq:4} in Lemma \ref{reduc} (on $\mathcal{N}(\sigma_{j+1},\ldots,\sigma_m)$) and \eqref{partida}, we obtain that
\[0 = \int x^{\nu} \mathcal{A}_{n,j+1}(x) \frac{d \sigma_{j+1}(x)}{w_{n,j}(x)}, \qquad \nu = 0,\ldots,n-1.\]
These orthogonality relations imply that $\mathcal{A}_{n,j+1}$ has at least $n$ sign changes on the interval $\Delta_{j+1}$. Let $w_{n,j+1}$ be a polynomial of degree $n$ with simple zeros at points of $\Delta_{j+1}$ where $\mathcal{A}_{n,j+1}$ changes sign; therefore, its zeros belong to $\mathbb{C} \setminus \Delta_{j+2}$. By the way in which $w_{n,j+1}$ was defined, we have
\[\frac{\mathcal{A}_{n,j+1}(z)}{w_{n,j+1}(z)}\in \mathcal{H}(\mathbb{C} \setminus \Delta_{j+2}).\]
Taking into account that $\deg w_{n,j+1} = n$ and since $\mathcal{A}_{n,j+1}(z) = \mathcal{O}(1/z)$ according to Definition \ref{MTPm2}, we also have ${\mathcal{A}_{n,j+1}(z)}/{w_{n,j+1}(z)} = \mathcal{O}(1/z^{n+1})$ as claimed.

\medskip

For $j=m-1$ equation  (\ref{partida}) takes the form
\begin{equation}\label{MP}
\frac{\left(a_{n,m-1}-a_{n,m}\widehat{s}_{m,m}\right)(z)}{w_{n,m-1}(z)} =\mathcal{O}(1/z^{n+1}) \in \mathcal{H}(\mathbb{C} \setminus \Delta_{m}).
\end{equation}
Using \eqref{eq:4}, we obtain
\begin{equation}\label{orto}
\int x^\nu a_{n,m}(x) \frac{d s_{m,m}(x)}{w_{n,m-1}(x)} = 0, \qquad \nu =0,\ldots,n-1.
\end{equation}

If $\deg a_{n,m} < n$, \eqref{orto} implies that $a_{n,m} \equiv 0$ since it would be orthogonal to itself (and the measure has constant sign); otherwise, \eqref{orto} implies that $a_{n,m}$ has at least $n$ sign changes on $\Delta_m$ and since $\deg a_{n,m} \leq n$ it must have exactly $n$ simple zeros inside $\Delta_m$. With this we conclude the proof of the lemma. \end{proof}

This lemma shows that $a_{n,m}$ is uniquely determined except for a constant factor (or it is identically equal to zero) because otherwise by linearity we could construct an $a_{n,m}$, not identically equal to zero, of degree smaller than $n$.

\medskip

We need some formulas verified by the Cauchy transforms of products of measures in a Nikishin system.  It is  known that for each $\sigma
\in {\mathcal{M}}(\Delta),$ where $\Delta$ is an infinite subinterval of the real line different from $\mathbb{R}$, there exists a measure $\tau \in
{\mathcal{M}}(\Delta)$ and ${\ell}(z)=a z+b, a = 1/|\sigma|, b \in {\mathbb{R}},$ such that
\begin{equation} \label{s22}
{1}/{\widehat{\sigma}(z)}={\ell}(z)+ \widehat{\tau}(z),
\end{equation}
where $|\sigma|$ is the total variation of the measure $\sigma.$  See  \cite[Appendix]{KN} for bounded $\Delta$, and \cite[Lemma 2.3]{FL4} when $\Delta$ is unbounded.
In particular,
\[
{1}/{\widehat{s}_{1,1}(z)}  ={\ell}_{1,1}(z)+
\widehat{\tau}_{1,1}(z).
\]
Sometimes we write $\langle
\sigma_{\alpha},\sigma_{\beta} \widehat{\rangle}$ in place of $\widehat{s}_{\alpha,\beta}$.
In \cite[Lemma 2.10]{FL4}, several formulas involving ratios of Cauchy transforms were proved. The most useful ones  in this paper establish that
\begin{equation} \label{4.4}
\frac{\widehat{s}_{1,k}}{\widehat{s}_{1,1}} =
\frac{|s_{1,k}|}{|s_{1,1}|} - \langle \tau_{1,1},\langle s_{2,k},s_{1,1}
\rangle \widehat{\rangle}  , \qquad  1=j < k \leq m.
\end{equation}

\medskip

Another important ingredient in the proof of Theorem \ref{CTPm} is the notion of convergence in Hausdorff content. Let $B$ be a subset of the complex plane $\mathbb{C}$. By
$\mathcal{U}(B)$ we denote the class of all coverings of $B$ by at
most a numerable set of disks. Set
$$
h(B)=\inf\left\{\sum_{i=1}^\infty
|U_i|\,:\,\{U_i\}\in\mathcal{U}(B)\right\},
$$
where $|U_i|$ stands for the radius of the disk $U_i$. The quantity
$h(B)$ is called the $1$-dimensional Hausdorff content of the
set $B$.

\medskip

Let $(\varphi_n)_{n\in\mathbb{N}}$ be a sequence of complex functions
defined on a domain $D\subset\mathbb{C}$ and $\varphi$ another
function defined on $D$ (the value $\infty$ is permitted). We say that
$(\varphi_n)_{n\in\mathbb{N}}$ converges in Hausdorff content to
the function $\varphi$ inside $D$ if for each compact
subset $\mathcal{K}$ of $D$ and for each $\varepsilon
>0$, we have
\begin{equation}
\label{convH}
\lim_{n\to\infty} h\{z\in K :
|\varphi_n(z)-\varphi(z)|>\varepsilon\}=0
\end{equation}
(by convention $\infty \pm \infty = \infty$). We denote this writing $h$-$\lim_{n\to \infty} \varphi_n =
\varphi$ inside $D$.

\medskip

If the functions $\varphi_n$ are holomorphic in $D$ and \eqref{convH} takes place, then the convergence is uniform on each compact subset of $D$. This result is proved in \cite[Lemma 1]{Gon}. Therefore, in order to prove \eqref{Con01m}-\eqref{Con00m} it is sufficient to show that the convergence takes place in Hausdorff content in the corresponding region. This is what we will do.

\subsection{Proof of Theorem \ref{equiv}.}
\begin{proof}
\noindent Fix $n \in \mathbb{N}$. Notice that relations \eqref{tipoIm} and \eqref{tipoIIam} are the same as relations \eqref{tipoIam} and \eqref{tipoIdm}, respectively.
First, let us show that Definition \ref{MTPm2}  implies Definition \ref{MTPm1}. Using \eqref{eq:aux} with $\mathcal{L}_j = \mathcal{A}_{n,j}, j \in \{0,1,\ldots,m-2\},$ and $r = m-1$, we obtain
\[\mathcal{A}_{n,j} + \sum_{k=j+1}^{m-1} (-1)^{k-j}\widehat{s}_{k,j+1} \mathcal{A}_{n,k} = (-1)^j \left(a_{n,j} -  a_{n,m} \widehat{s}_{m,j+1}\right).
\]
However, \eqref{tipoIam}-\eqref{tipoIdm} imply that the left hand side of this equation is $\mathcal{O}(1/z)$ and, therefore, so is the right hand side which is exactly the expression that appears in Definition \ref{MTPm1}. We also obtain the additional relation
\begin{equation}
\label{additional}
(a_{n,0} - a_{n,m}\widehat{s}_{m.1})(z) = \mathcal{O}(1/z).
\end{equation}
which is redundant with respect to the  equations in Definition \ref{MTPm1}, but \eqref{additional} will be needed below.

\medskip

To prove the converse, we observe that according to formula \eqref{rel-fund} applied to the Nikishin system $\mathcal{N}(\sigma_{j+1}, \sigma_{j+2},\ldots, \sigma_{m})$, for $j=0,\ldots,m-2,$ we have
\begin{equation}\label{RM}
 0\equiv (-1)^{j }\widehat{s}_{m,j+1}  + \sum_{k=j+1}^{m-1}(-1)^{k}   \widehat{s}_{m,k+1} \widehat{s}_{j+1,k} +
 (-1)^{m} \widehat{s}_{j+1, m}, \quad z \in \mathbb{C} \setminus (\Delta_{j+1} \cup \Delta_m).
\end{equation}

From Definition \ref{MTPm1} let us consider  the following equations
\begin{align}
(-1)^{j}\left(a_{n,j}-a_{n,m}\widehat{s}_{m,j+1}\right)(z)=\mathcal{O}(1/z)\label{E4},\\
..........................................................\nonumber\\
(-1)^{m-2}\left(a_{n,m-2}-a_{n,m}\widehat{s}_{m,m-1}\right)(z)=\mathcal{O}(1/z)\label{E2},\\
(-1)^{m-1}\left(a_{n,m-1}-a_{n,m}\widehat{s}_{m,m}\right)(z)=\mathcal{O}(1/z)\label{E1}.
\end{align}

Let us multiply equation (\ref{E1}) by $\widehat{s}_{j+1,m-1}$, (\ref{E2}) by $\widehat{s}_{j+1,m-2}$, and so on until we arrive to (\ref{E4}) multiplied by $1$. Adding up all the relations so obtained, we arrive at
\begin{align}
(-1)^{j} a_{n,j}+(-1)^{j+1}a_{n,j+1}\widehat{s}_{j+1,j+1}+\ldots + (-1)^{m-1}a_{n,m-1} \widehat{s}_{j+1,m-1} \nonumber\\
-a_{n,m}\left((-1)^{j}\widehat{s}_{m,j+1}+\sum_{k=j+1}^{m-1}(-1)^{k}   \widehat{s}_{m,k+1} \widehat{s}_{j+1,k}\right)=\mathcal{O}(1/z).
\end{align}
Using relation (\ref{RM}) to replace what is inside the big parenthesis,  the   $j$-th equation in Definition \ref{MTPm2} follows immediately; that is
$\mathcal{A}_{n,j}=\mathcal{O}(1/z).$
Therefore, DR and ML Hermite-Pad\'e polynomials coincide.

\medskip

Observe that \eqref{additional} and \eqref{tipoIIcm}-\eqref{tipoIIam} imply that once $a_{n,m}$ is found then $a_{n,j}, j=0,\ldots,m-1,$ is uniquely determined as the polynomial part of the asymptotic expansion at $\infty$ of $a_{n,m}\widehat{s}_{m,j+1}$. In particular, this means that if $a_{n,j}^{n-1},$ is the coefficient corresponding to the power $z^{n-1}$ of  $a_{n,j}$ and
$a_{n,m}^{n}$ the coefficient corresponding to the power $z^{n}$ of  $a_{n,m}$ then
		$$a_{n,j}^{n-1} = a_{n,m}^{n}|s_{m,j+1}|, \qquad j=1,\ldots,m-1. $$
From Lemma \ref{degree} we know that $\deg a_{n,m} = n$, therefore, $\deg a_{n,j} = n-1, j=0,\ldots,m-1$ and the polynomials $a_{n,j},j=0,\ldots,m-1$ are uniquely defined up to a constant factor (the same constant as for $a_{n,m}$). Moreover, \eqref{MP} implies that  $\frac{a_{n,m-1}}{a_{n,m}}$ is an $n$-th diagonal multipoint Pad\'e approximation of
$\widehat{s}_{m,m}$ with $n+1$ interpolation conditions at $\infty$ and another $n$ located at the zeros of $w_{n,m-1}$ and $a_{n,m-1}$ is the $n$-th polynomial of the second lind with respect to the measure $d\sigma_m/w_{n,m-1}$ whose $n-1$ zeros are known to interlace the $n$ simple zeros of $a_{n,m}$.
With this we conclude the proof. \end{proof}

The property of the degrees of the polynomials $a_{n,j}, j=0,\ldots,m,$ indicates that the multi-indices $(n,\ldots,n,n+1) \in \mathbb{Z}_+^m \setminus \{{\bf 0}\}$ are normal (for the definition of normality see, for example, the introduction in \cite{FL3}). This is emphasized in the next result.

\begin{corollary} \label{misc}
For each fixed $n\in \mathbb{N}$ we have
\begin{enumerate}
\item[(a)] $\mathcal{A}_{n,0}$ has no zero in $\mathbb{C} \setminus \Delta_1$. The coefficient accompanying $1/z^{n+1}$ in the asymptotic expansion \eqref{tipoIam} is different from zero.
\item[(b)] $\mathcal{A}_{n,j}, j=1,\ldots,m$ has exactly $n$ zeros in $\mathbb{C} \setminus \Delta_{j+1}\,\,(\Delta_{m+1} = \emptyset)$, they are all simple and lie in $\stackrel{\circ}\Delta_j$. The coefficient accompanying $1/z$ in the asymptotic expansions \eqref{tipoIbm}-\eqref{tipoIdm} is different from zero.
\end{enumerate}
\end{corollary}
\begin{proof} The forms $\mathcal{A}_{n,j}$ are symmetric with respect to the real line (that is $\mathcal{A}_{n,j}(\overline{z}) = \overline{\mathcal{A}_{n,j}( {z})}$); therefore, its non real zeros come in conjugate pairs. For $j=m, \mathcal{A}_{n,m}= a_{n,m}$ and the property stated in $(b)$ about the zeros was proved in Theorem \ref{equiv}.  Suppose that for some $\overline{\j} \in \{0,\ldots,m-1\}$ any one of the properties stated in $(a)$ or $(b)$ fails. From the proof of Lemma \ref{degree}, we know that $\mathcal{A}_{n,\overline{\j}}, \overline{\j} = 1,\ldots,m-1,$ has at least $n$ sign changes on $\Delta_{\overline{\j}}$, so it can have more but not less than $n$ zeros in $\mathbb{C} \setminus \Delta_{\overline{\j}}$. Thus, there exists a polynomial $w_{n,\overline{\j}}$ of degree $\geq n$ or $0$ if $\overline{\j} = 0$, with real coefficients, such that
\begin{equation}\label{extra}\mathcal{A}_{n,\overline{\j}}/w_{n,\overline{\j}} = \mathcal{O}(1/z^{n+2}) \in \mathcal{H}(\mathbb{C}\setminus \Delta_{\overline{\j}}).
\end{equation}
Arguing as in the proof of Lemma \ref{degree} it readily follows  that for $j = \overline{\j},\ldots,m-1$ one also has \eqref{extra}. This entails that for some polynomial $w_{n,m-1}$, with real coefficients,
\[
\int x^\nu a_{n,m}(x) \frac{d s_{m,m}(x)}{w_{n,m-1}(x)} = 0, \qquad \nu =0,\ldots,n.
\]
This implies that $a_{n,m} \equiv 0$. This is impossible because according to Theorem \ref{equiv} we would also have $a_{n,j} \equiv 0, j=0,\ldots,m-1$. Thus, all properties stated in this corollary must hold.
\end{proof}

\subsection{Proof of Theorem \ref{TBIO}.}

\begin{proof}
	For $k=0,\ldots,m-1$, set
\[\mathcal{B}_{k,j}(z) = (-1)^j b_{k,j}+ (-1)^{j+1}b_{k,j+1}\widehat{s}_{m-j,m-j}\cdots (-1)^{m}b_{k,m}\widehat{s}_{m-j,1} \]
First, let us analyze the case $n < k$. From equation \eqref{dtipoIam} and  \eqref{eq:4} it follows that
	for $ \nu=0,1,\ldots k-1$
	\begin{equation} \label{bio1}
	0=\int_{\Delta_m} x_m^{\nu}\,\mathcal{B}_{k,1}(x_m)d\sigma_{m}(x_m).
	\end{equation}
Using consecutively	\eqref{dtipoIbm}-\eqref{dtipoIdm} and \eqref{eq:3}, we have
	\begin{equation} \label{bio2}
\mathcal{B}_{k,1}(x_m) =
\int_{\Delta_{m-1}} \mathcal{B}_{k,2}(x_{m-1}) \frac{d\sigma_{m-1}(x_{m-1})}{(x_m-x_{m-1})} = \cdots =
\end{equation}
\[
\int_{\Delta_{m-1}}\ldots\int_{\Delta_2} \mathcal{B}_{k,m-1}(x_{2}) \frac{d\sigma_{2}(x_{2}) }{(x_{3}-x_{2})}\cdots\frac{d\sigma_{m-1}(x_{m-1}) }{(x_{m}-x_{m-1})} =
\]
\[
 \int_{\Delta_{m-1}}\cdots\int_{\Delta_1} (-1)^m b_{k,m}(x_{1}) \frac{d\sigma_{1}(x_{1}) }{(x_{2}-x_{1})}\cdots\frac{d\sigma_{m-1}(x_{m-1}) }{(x_{m}-x_{m-1})} = - \int_{\Delta_1} b_{k,m}(x_{1}) K(x_1,x_m) d\sigma_1(x_1).
\]
In the last equality, we use Fubini's theorem and the definition of the kernel $K(x_1,x_m)$. Combining \eqref{bio1}, \eqref{bio2} and using the fact that $n < k$, we get
$$
\int_{\Delta_m}\int_{\Delta_1} a_{n,m}(x_m)K(x_1,x_m)b_{k,m}(x_1)  d\sigma_1(x_1) d\sigma_m(x_m) =0, \qquad n < k.
$$
For $k < n$, the proof is the same as above applied to the forms $\mathcal{A}_{n,j}$ instead of the forms $\mathcal{B}_{k,j}$.

Now, suppose that
\begin{equation}\label{bio3}
\int_{\Delta_m}\int_{\Delta_1} a_{n,m}(x_m)K(x_1,x_m)b_{n,m}(x_1)  d\sigma_1(x_1) d\sigma_m(x_m) = 0.
\end{equation}
Obviously, $\deg b_{k,m} = k, k \geq 0$. This, together with the orthogonality relations and  \eqref{bio3} give
\begin{equation}\label{bio4}
\int_{\Delta_1} x_1^n \int_{\Delta_m}   a_{n,m}(x_m)K(x_1,x_m)  d\sigma_m(x_m) d\sigma_1(x_1) = 0.
\end{equation}
On the other hand, just as in the proof of \eqref{bio2} we have
\begin{equation}\label{bio5}
\mathcal{A}_{n,1}(x_1) = \int_{\Delta_m}   a_{n,m}(x_m)K(x_1,x_m)  d\sigma_m(x_m).
\end{equation}
Now, from \eqref{bio4} and \eqref{bio5} we obtain
\[\int_{\Delta_1} x_1^n \mathcal{A}_{n,1}(x_1) d\sigma_1(x_1) = \int_{\Delta_1} x_1^n \int_{\Delta_m}   a_{n,m}(x_m)K(x_1,x_m)  d\sigma_m(x_m) = 0.
\]
We also know that
\[\int_{\Delta_1} x_1^\nu \mathcal{A}_{n,1}(x_1) d\sigma_1(x_1) =  0,\qquad \nu=0,\ldots,n-1.
\]
However, these orthogonality relations imply that  the form $\mathcal{A}_{n,1}$ has at least $n+1$ sign changes on $\Delta_1$ but this is impossible since in Corollary \ref{misc} we showed that it has  exactly $n$ zeros in $\mathbb{C} \setminus \Delta_2$. Therefore,
\[\int_{\Delta_m}\int_{\Delta_1} a_{n,m}(x_m)K(x_1,x_m)b_{n,m}(x_1)  d\sigma_1(x_1) d\sigma_m(x_m) \neq 0
\]
and we conclude the proof.
\end{proof}

\subsection{Proof of Theorem \ref{CTPm}.}
\begin{proof}
In the proof of Theorem \ref{equiv} it was indicated that ${a_{n,m-1}}/{a_{n,m}}$ is an $n$-th diagonal multipoint Pad\'e approximation of
$\widehat{s}_{m,m}$. Since either the sequence of moments of $\sigma_m$ satisfies Carleman's condition or $\Delta_{m-1}$ is a finite interval contained in $\mathbb{C} \setminus \Delta_m$, using \cite[Theorem 1]{lago} we have
\begin{equation} \label{multipoint}
\lim_{n\rightarrow \infty}  \frac{a_{ n,m-1}}{a_{ n,m}} =  \widehat{s}_{m,m},
\end{equation}
uniformly on each compact subset $\mathcal{K} \subset \mathbb{C} \setminus \Delta_m$.

\medskip

Now,  consider the case $j \in \{0,\ldots,m-2\}$. Having in mind \eqref{Con01m}, we need to reduce $\mathcal{A}_{ n,j}$ so as to eliminate all $a_{n,k}, k=j+1,\ldots,m-1$. We start out eliminating $a_{n,j+1}$. Consider the ratio $\mathcal{A}_{ n,j}/\widehat{s}_{j+1,j+1}$. Using \eqref{s22} and \eqref{4.4} we obtain that
\[ \frac{\mathcal{A}_{ n,j}}{\widehat{s}_{j+1,j+1}} =    \left((-1)^j p_{j+1,j+1} a_{ n,j}+ \sum_{k=j+1}^m \frac{(-1)^k|s_{j+1,k}|}{|s_{j+1,j+1}|} a_{n,k} \right) + (-1)^ja_{n,j}\widehat{\tau}_{j+1,j+1} -
\]
\[\sum_{k=j+2}^m (-1)^ka_{n,k} \langle  {\tau}_{j+1,j+1}, \langle s_{j+2,k}, s_{j+1,j+1} \rangle \widehat{\rangle},
\]
has the form of $\mathcal{L}_0$ in Lemma \ref{reduc} (with respect to $\mathcal{N}(\tau_{j+1,j+1},s_{j+2,j+1},\sigma_{j+3},\ldots,\sigma_m)$, when $j+3\leq m$, or $\mathcal{N}(\tau_{j+1,j+1},s_{j+2,j+1})$ if $j=m-2$). Notice that $ {\mathcal{A}_{ n,j}}/({\widehat{\sigma}_{j+1}w_{n,j}}) \in \mathcal{H}(\mathbb{C} \setminus \Delta_{j+1})$, and
\[\frac{\mathcal{A}_{ n,j}}{\widehat{s}_{j+1,j+1}w_{n,j}} \in \mathcal{O}\left(\frac{1}{z^{n}}\right), \qquad z \to \infty.
\]
From \eqref{eq:4} of Lemma \ref{reduc}, we obtain that for $\nu = 0,\ldots,n-2$
\[ 0 = \int x^{\nu} \left( (-1)^ja_{ n,j}(x) - \sum_{k=j+2}^m (-1)^k a_{ n,k}  \langle s_{j+2,k}, s_{j+1,j+1}  \widehat{\rangle}(x) \right)\frac{d\tau_{j+1,j+1}(x)}{w_{n,j}(x)}
\]
which implies that the function in parenthesis under the integral sign has at least $n-1$ sign changes in $\stackrel{\circ}{\Delta}_{j+1} $. In turn, it follows that there exists  a polynomial $\widetilde{w}_{ n,j }, \deg \widetilde{w}_{ n,j } = n-1$, whose zeros are simple and lie in $\stackrel{\circ}{\Delta}_{j+1} $ such that
\begin{equation}
 \label{eq:5}
 \frac{(-1)^ja_{ n,j}  - \sum_{k=j+2}^m (-1)^ka_{ n,k}  \langle s_{j+2,k}, s_{j+1,j+1}  \widehat{\rangle} }{\widetilde{w}_{n,j }} \in \mathcal{H}(\mathbb{C} \setminus \Delta_{j+2}).
\end{equation}

\medskip

On the other hand, using \eqref{eq:aux} with $r=j+1$ and Definition \ref{MTPm2}, we obtain
\begin{align*}
(\mathcal{A}_{ n,j}-\widehat{s}_{j+1,j+1}\mathcal{A}_{ n,j+1})(z)=\mathcal{O}\left(\frac{1}{z}\right)\\
=\left((-1)^ja_{ n,j}  - \sum_{k=j+2}^m (-1)^k a_{ n,k}  \langle s_{j+2,k}, s_{j+1,j+1}  \widehat{\rangle}\right)(z).
\end{align*}
Consequently,
\begin{align} \label{eq:6}
 \frac{(-1)^ja_{n,j}  - \sum_{k=j+2}^m (-1)^k a_{ n,k}  \langle s_{j+2,k}, s_{j+1,j+1}  \widehat{\rangle} }{\widetilde{w}_{n,j }} = \mathcal{O}\left(\frac{1}{z^{n}}\right), \qquad z \to \infty.
\end{align}
Notice that $\langle s_{j+2,k}, s_{j+1,j+1}{\rangle} = s_{j+2,j+1} $ when $k=j+2$ and $\langle s_{j+2,k}, s_{j+1,j+1}{\rangle} = \langle s_{j+2,j+1},s_{j+3,k}\rangle$ when $j+3\leq k \leq m$ (if any).

\medskip

Suppose that $j=m-2$. In this case, \eqref{eq:5}-\eqref{eq:6} reduce to
\[ \frac{a_{ n,m-2}  -  a_{ n,m} \widehat{s}_{m,m-1} }{\widetilde{w}_{n,m-2 }} = \mathcal{O}\left(\frac{1}{z^{n}}\right) \in \mathcal{H}(\mathbb{C} \setminus \Delta_{m}). \]
In comparison with the case when $j=m-1$ we lose one interpolation condition at infinity and we say that ${a_{n,m-2}}/{a_{n,m}}$ is an incomplete diagonal Pad\'e approximation of $s_{m,m-1}$. However, using \cite[Lemma 2]{Bus} we can assert that
\[ h- \lim_{n\to \infty} \frac{a_{n,m-2}}{a_{n,m}} = \widehat{s}_{m,m-1}\]
inside $\mathbb{C} \setminus \Delta_m$. Since the poles of ${a_{n,m-2}}/{a_{n,m}}$ lie in $\Delta_m$, from \cite[Lemma 1]{Gon}  uniform convergence on compact subsets of $\mathbb{C} \setminus \Delta_{m}$ readily follows.

\medskip

Incidentally, if $m=2$ and $j=0$, on the right hand side of \eqref{eq:6} we have $\mathcal{O}(1/z^{n+1})$ because $\mathcal{A}_{n,0} - \widehat{s}_{1,1} \mathcal{A}_{n,1} = \mathcal{O}(1/z^2)$. So, in this case $a_{n,0}/a_{n,2}$ is a complete multipoint Pad\'e approximation of $\widehat{s}_{2,1}$. Then $a_{n,2}$ is (also) an $n$-th orthogonal polynomial with respect to the varying measure $d s_{2,1}/\widetilde{w}_{n,0}$ and $a_{n,0}$ is the associated polynomial of second kind which implies that   $a_{n,0}$ has $n-1$ simple zeros which lie in $ \stackrel{\circ}\Delta_m$ (and interlace the zeros of $a_{n,2}$). For other values of $m$, we discuss later the degree and location of the zeros of $a_{n,j}, j=0,\ldots,m-2$.

\medskip

Let us assume that $m \geq 3$ and $0 \leq j \leq m-3$. Then, $\langle s_{j+2,k}, s_{j+1,j+1} \rangle = \langle s_{j+2,j+1} , s_{j+3,k} \rangle, k=j+3,\ldots,m,$ and we use this equality to modify the corresponding terms in the numerators of the left hand sides of \eqref{eq:5} and \eqref{eq:6} which becomes
\[(-1)^j a_{n,j}  - (-1)^{j+2} a_{n,j+2}\widehat{s}_{j+2,j+1} - \sum_{k=j+3}^m (-1)^k a_{n,k}  \langle  s_{j+2,j+1} , s_{j+3,k} \widehat{\rangle}.\]
Now, we must do away with $a_{n,j+2}$. Using \eqref{s22} and \eqref{4.4}, we obtain
\[ \frac{(-1)^j a_{n,j}  - (-1)^{j+2} a_{n,j+2}\widehat{s}_{j+2,j+1} - \sum_{k=j+3}^m (-1)^k a_{n,k}  \langle  s_{j+2,j+1} , s_{j+3,k}  \widehat{\rangle}}{\widehat{s}_{j+2,j+1}} =
\]
\begin{align*}
\left((-1)^j p_{j+2,j+1} a_{ n,j} - (-1)^{j+2} a_{n,j+2}  - \sum_{k=j+3}^m \frac{(-1)^k|\langle  s_{j+2,j+1} , s_{j+3,k} \rangle|}{|s_{j+2,j+1}|} a_{n,k} \right)  \\
+ (-1)^j a_{n,j}\widehat{\tau}_{j+2,j+1} + (-1)^2 \sum_{k=j+3}^m (-1)^k a_{n,k} \langle  {\tau}_{j+2,j+1}, \langle s_{j+3,k}, s_{j+2,j+1} \rangle \widehat{\rangle}.
\end{align*}
This expression has the form of $\mathcal{L}_0$ in Lemma \ref{reduc}. Additionally,
$$ \frac{(-1)^j a_{n,j}  - (-1)^{j+2} a_{n,j+2}\widehat{s}_{j+2,j+1} - \sum_{k=j+3}^m (-1)^k a_{n,k}  \langle  s_{j+2,j+1} , s_{j+3,k} \widehat{\rangle}}{\widehat{s}_{j+2,j+1}\widetilde{w}_{n,j }}\in \mathcal{H}(\mathbb{C} \setminus \Delta_{j+2})$$
and because of \eqref{eq:6}, as $z\to \infty,$
$$ \frac{(-1)^j a_{n,j}  - (-1)^{j+2} a_{n,j+2}\widehat{s}_{j+2,j+1} - \sum_{k=j+3}^m (-1)^k a_{n,k}  \langle  s_{j+2,j+1} , s_{j+3,k} \widehat{\rangle}}{\widehat{s}_{j+2,j+1}\widetilde{w}_{n,j }}=\mathcal{O}\left(\frac{1}{z^{n-1}}\right).$$

From \eqref{eq:4} in Lemma \ref{reduc}, we obtain that for $\nu = 0,\ldots,n-3$
\[ 0 = \int x^{\nu} \left((-1)^j a_{ n,j} + (-1)^2 \sum_{k=j+3}^m (-1)^k a_{ n,k}(x) \langle s_{j+3,k}, s_{j+2,j+1}   \widehat{\rangle}(x) \right)\frac{d\tau_{j+2,j+1}(x)}{\widetilde{w}_{n,j}(x)}
\]
which implies that the function in parenthesis under the integral sign has at least $n-2$ sign changes in $\stackrel{\circ}{\Delta}_{j+2}$. In turn, it follows that there exists  a polynomial $\widetilde{w}_{ n,j+1 }, \deg \widetilde{w}_{ n,j+1 } = n-2$, whose zeros are simple and lie in $\stackrel{\circ}{\Delta}_{j+1} $ such that
\[ \frac{(-1)^ja_{ n,j} + (-1)^2 \sum_{k=j+3}^m (-1)^k a_{ n,k} \langle s_{j+3,k}, s_{j+2,j+1} \rangle \widehat{\rangle}}{\widetilde{w}_{n,j+1 }} \in \mathcal{H}(\mathbb{C} \setminus \Delta_{j+3})
\]
On the other hand, according to Definition \ref{MTPm2} and Lemma \ref{lem:2} with $r=j+2$
\begin{align*}
(\mathcal{A}_{ n,j}-\widehat{s}_{j+1,j+1}\mathcal{A}_{ n,j+1}+\widehat{s}_{j+2,j+1}\mathcal{A}_{ n,j+2})(z)=\mathcal{O}\left(\frac{1}{z}\right)\\
=\left((-1)^ja_{ n,j}  + (-1)^2 \sum_{k=j+3}^m (-1)^k a_{ n,k}  \langle   s_{j+3,k}, s_{j+2,j+1}  \widehat{\rangle}\right)(z)
\end{align*}
Consequently,
\begin{align*}
\frac{(-1)^j a_{ n,j} + (-1)^2\sum_{k=j+3}^m (-1)^k a_{ n,k} \langle s_{j+3,k}, s_{j+2,j+1} \rangle \widehat{\rangle}}{\widetilde{w}_{n,j+1 }}=\mathcal{O}\left(\frac{1}{z^{n-1}}\right)
\end{align*}
Notice that $a_{ n,j+2}$ has been eliminated. If $j=m-3$ combining \cite[Lemma 2]{Bus} and \cite[Lemma 1]{Gon} we obtain that
\[\lim_{n\to \infty} \frac{a_{n,m-3}}{a_{n,m}} = \widehat{s}_{m,m-2}\]
uniformly on each compact subset of $\mathbb{C}\setminus \Delta_m.$
Otherwise, we have
\[\frac{(-1)^ja_{ n,j} + (-1)^2 a_{n,j+3}\widehat{s}_{j+3,j+1} + (-1)^2 \sum_{k=j+4}^m (-1)^k a_{ n,k} \langle s_{j+3,j+1}, s_{j+1,k} \rangle \widehat{\rangle}}{\widetilde{w}_{n,j+1 }}=\mathcal{O}\left(\frac{1}{z^{n-1}}\right),\]
and we are ready to eliminate $a_{n,j+3}$ dividing by $\widehat{s}_{j+3,j+1}$.

\medskip

In general, for fixed $j$, after $m-j-1$ reductions obtained applying Lemmas \ref{reduc} and \ref{lem:2}, we find that there exists a polynomial denoted $w_{ n,j}^*, \deg w_{ n,j}^* = n-m+j$, whose zeros are simple and lie in $\stackrel{\circ}{\Delta}_{m-1} $ such that
\begin{equation}
\label{last}
\frac{a_{ n,j} - a_{ n,m} \widehat{s}_{m,j+1}}{w_{n,j}^*} = \mathcal{O}\left(\frac{1}{z^{n-m+j+2}}\right) \in \mathcal{H}(\mathbb{C} \setminus \Delta_m), \qquad z \to \infty.
\end{equation}
Then, \cite[Lemma 2]{Bus} and \cite[Lemma 1]{Gon}, imply \eqref{Con01m}. Now, \eqref{Con00m} is an immediate consequence of \eqref{RM} and \eqref{Con01m}.

\medskip

Assuming that $j \in \{0,\ldots,m-1\}$, \eqref{last} implies that
\[\frac{a_{ n,j} - a_{ n,m} \widehat{s}_{m,j+1}}{\widehat{s}_{m,j+1} w_{n,j}^*} = \mathcal{O}\left(\frac{1}{z^{n-m+j+1}}\right) \in \mathcal{H}(\mathbb{C} \setminus \Delta_m), \qquad z \to \infty.\]
Since
\[\frac{a_{ n,j} - a_{ n,m} \widehat{s}_{m,j+1}}{\widehat{s}_{m,j+1}} = a_{ n,j} \widehat{\tau}_{m,j+1} - \left(a_{n,m} - \ell_{m,j+1}a_{n,j}\right),\]
and using \eqref{eq:4}
\[\int x^{\nu} a_{n,j}(x)  \frac{d\tau_{m,j+1}(x)}{w_{n,j}^*(x)} = 0, \qquad \nu = 0,\ldots, n-m+j-1.\]
Consequently, $a_{n,j}$ has at least $n-m+j$ sign changes in $\stackrel{\circ}{\Delta}_m$; therefore, $\deg a_{n,j} \geq n-m+j$. For $j \in \{0,\ldots,m-2\}$ it could occur that $m-j-1$ zeros of $a_{n,j}$ lie in $\mathbb{C} \setminus \Delta_m$.
\end{proof}

Although, there may be a certain amount (independent of $n$) of zeros of the polynomials $a_{n,0},\ldots,a_{n,m-2}$ that abandon $\Delta_m$, the next corollary shows that in that case they approach $\Delta_m$ as $n\to \infty$.

\begin{corollary}\label{cor:CTPm} Under the assumptions of Theorem \ref{CTPm} we have that the accumulation points of  the zeros of $a_{ n,j}, j=0,1,\cdots,  m-2$  are in $\Delta_m $.
\end{corollary}

\begin{proof}
Let $\Gamma$ be an arbitrary simple closed Jordan curve contained in $\mathbb{C} \setminus \Delta_m$. Since $\widehat{s}_{m,j+1}$ is never equals zero on  this domain, the argument principle implies that
		\[\lim_{n\to  \infty}\frac{1}{2\pi i}\int_\Gamma \frac{\left(a_{n,j}(z)/a_{n,m}(z)\right)^\prime}{\left(a_{n,j}(z)/a_{n,m}(z)\right)} dz = \frac{1}{2\pi i}\int_\Gamma \frac{ \widehat{s}^{\prime}_{m,j+1}(z)}{\widehat{s}_{m,j+1}(z)} = 0.\]
		But the poles of $a_{n,j}/a_{n,m}$ all lie in $\Delta_m$; consequently, for all sufficiently large $n$ the zeros of these rational function must lie in the unbounded connected component of the complement of $\Gamma$. This means that as $n \to \infty$ the zeros of $a_{n,j}$ that may lie in $\mathbb{C} \setminus \Delta_m$ must accumulate on $\Delta_m$.
\end{proof}

\appendix

\section{The inverse cubic string and Hermite-Pad\'e approximation problems.} \label{inverse}

The goal of this Appendix is to show the connection between the inverse cubic string problem \eqref{CSP} and the Hermite-Pad\'e approximation  problem \ref{Nik2}.
Let us consider two solutions $\phi_1(y), \phi_2(y)$ to the initial
value problem associated to the boundary value problem \eqref{CSP} for the values of the spectral parameter $z$ and $w$, respectively.
Then, by elementary integration by parts,
\begin{equation*}
\int_{-1}^1 (\phi_{1,yyy} \phi_2+\phi_1 \phi_{2, yyy})(y)dy=B(\phi_1,\phi_2)(y)\big|_{-1}^1,
\end{equation*}
where $B(h_1, h_2)(y)=\left( h_{1, yy}h_2-h_{1,y}h_{2,y}+h_1h_{2, yy}\right) (y)$ is the \textit{bilinear concomitant} associated to the formally  skew-adjoint operator  $\frac{d^3}{dy^3}$.  For the cubic string, and the specified initial conditions at $y=-1$,
	this identity simplifies further to
	\begin{equation} \label{eq:bidentity}
	(z+w)\int_{-1}^1 (\phi_1 \phi_2 g)(y) dy=B(\phi_1, \phi_2)(y)\big|_{-1}^1,
	\end{equation}
	implying that the bilinear concomitant vanishes for $w=-z$.
Denoting $\phi(y;z)=\phi_1(y;z)$ and $\phi(y;-z)=\phi_2(y;-z)$ we
	are led to an important identity
	\begin{equation} \label{eq:duality}
	\phi_{yy}(1;z)\phi(1;-z)-\phi_y(1;z)\phi_y(1;-z)+\phi(1;z)\phi_{yy}(1;-z)=0.
	\end{equation}

\medskip

Let us interpret this identity with the help of
the Weyl (sometimes also called the Weyl-Titchmarsh) functions
\begin{equation*}
W(z)=\frac{\phi_y(1;z)}{\phi(1;z)}, \qquad Z(z)=\frac{\phi_{yy}(1;z)}{\phi(1;z)},
\end{equation*}
which are matrix elements of the resolvent of the operator defined by \eqref{CSP}.  In terms of $W(z)$ and $Z(z)$, \eqref{eq:duality} can be written as
\begin{equation}\label{eq:W1W2}
Z(z)-W(z)W(-z)+Z(-z)=0,
\end{equation}
indicating a relation between these two Weyl functions.  In \cite{jacek}
	it was proven that $W(z), Z(z)$ admit the following spectral representations
	\begin{equation} \label{eq:spectralrep}
	\frac{W(z)}{z}=\int \frac{d\mu(x)}{z-x}, \qquad \frac{Z(z)}{z}=\int \frac{d\nu(x)}{z-x},
	\end{equation}
	where $d\nu(x)= (x\int \frac{d\mu(y)}{x+y}) \, d\mu(x)$.   The proof in \cite{jacek} is carried out when
	$g(y)$ is a finite, discrete measure, and in this special case $d\mu(x)=\sum_{j=0}^N b_j\delta_{z_j}, b_j>0 $ where $z_0=0$, and $z_j>0$ are
	the eigenvalues of the cubic string boundary value problem \eqref{CSP}.

\medskip

The approximation problem used in \cite{jacek}  to solve the
inverse problem for the cubic string consists in finding, for each $n \in \mathbb{N} = \{1,2,\ldots\},$  polynomials $(\widehat{P}_n,P_n, Q_n)$ with $\deg \widehat{P}_n\leq n$, $\deg P_n\leq n$, and  $\deg Q_n \leq n$,  not all identically equal to zero and normalized by the conditions $\widehat{P}_n(0)=0, \, \, P_n(0)=1$, which satisfy:
	\begin{align}
	\widehat{P}_n(z)-P_n(z)W(-z)+Q_n(z)Z(-z)&=\mathcal{O}(1/z^{n+1})\label{JAPI*},\\
	P_n(z)-Q_n(z)W(z)&=\mathcal{O}(1), \label{JAPIIa*}\\
	\widehat{P}_n(z)-Q_n(z)Z(z)&=\mathcal{O}(1).  \label{JAPIIb*}
	\end{align}

Due to (\ref{eq:W1W2}), equation (\ref{JAPIIb*}) is redundant and it suffices to consider equations (\ref{JAPI*}) and (\ref{JAPIIa*}).

\medskip

From  (\ref{eq:spectralrep}) it is clear
that there are two natural measures in the problem, $d\alpha(x)=
d\mu(x)$ and $d\beta(x)=x d\mu(x)$.  However, it is also
convenient to observe that in addition there are
two \textit{reflected with respect to the origin} measures
$d\alpha^*, d\beta^*$.  For example $d\alpha^*=\sum_{j=0}^N b_j\delta_{-z_j}$.  The advantage of using reflected measures is that
the final formulas have only one type of kernel, namely the Cauchy kernel $\frac{1}{x-y}$, rather than both the Cauchy and Stieltjes $\frac{1}{x+y}$ kernels.  In addition, having in mind a connection to Nikishin systems,
we point out that the supports of $d\beta$ and $d\alpha^*$ are disjoint.

\medskip

The spectral representations of $W$ and $Z$ imply
	\begin{align*}
	\frac{W(z)}{z}&=\int \frac{d\alpha(x)}{z-x}, &\frac{Z(z)}{z}&=\int \frac{d\beta(x)}{z-x}\int \frac{d\alpha^*(y)}{x-y},\\
	\frac{W(-z)}{z}&=\int \frac{d\alpha^*(y)}{z-x}, & \frac{-Z(-z)}{z}&=\int \frac{d\beta^*(y)}{z-y}\int \frac{d\alpha(x)}{y-x},
	\end{align*}
which allows one to rewrite the approximation problem using a notation more convenient for our purpose.

\medskip

To this end we set
	\begin{align*}
	d\lambda_{1}(x):=d\alpha^*(x), \quad \qquad  d\lambda_{2}(x):=d\beta(x)
	\end{align*}
	and
	\begin{align*}
	&\widehat{\lambda}_{1,1}(z):=\int \frac{d\lambda_1(x)}{z-x}, \quad &\widehat{\lambda}_{1,2}(z):=\int \frac{d\lambda_1(x)}{z-x}\int \frac{d\lambda_2(y)}{x-y} , \\
	&\widehat{\lambda}_{2,2}(z):=\int \frac{d\lambda_2(x)}{z-x}, \quad &\widehat{\lambda}_{2,1}(z):=\int\frac{d\lambda_2(x)}{z-x}\int\frac{d\lambda_1(y)}{x-y}.
	\end{align*}
	Therefore, we have the two Nikishin systems, namely, $(\lambda_{1,1},\lambda_{1,2})=\mathcal{N}(\lambda_1, \lambda_2)$ and its reverse $(\lambda_{2,2},\lambda_{2,1})=\mathcal{N}(\lambda_2, \lambda_1)$.

\medskip
	
	Now,  $W(z)=\int \frac{z d\alpha(x)}{z-x}$ and from equation (\ref{JAPIIa*})
    $P_{n}^{n} = Q_{n}^{n}|\alpha|$, where $P_{n}^{n}$ (respectively $Q_{n}^{n}$) is the coefficient in $P_n$ at the power $z^{n}$  and $|\alpha|$ is the zeroth moment of $d\alpha$. Then
    \begin{align*}
    W(z)=\int d\alpha(x) + \int \frac{xd\alpha(x)}{z-x}
    =|\alpha| + \int \frac{d\lambda_2(x)}{z-x}
    =|\alpha|+\widehat{\lambda}_{2,2}(z),
    \end{align*}
    implying that equation (\ref{JAPIIa*}) takes now the form
 \begin{align}\label{W}
    \widetilde{P}_n(z)-Q_n(z)\widehat{\lambda}_{2,2}(z)=\mathcal{O}(1)
    \end{align}
    where $\widetilde{P}_n =P_{n}  -|\alpha|Q_{n}$. Notice that $\deg \widetilde{P}_n \le n-1$.

 \medskip

In the next step we will rephrase equation (\ref{JAPI*}).  First, after some elementary computations,  we obtain
    $$W(-z)=z\widehat{\lambda}_{1,1}(z), \qquad Z(-z)=z\widehat{\lambda}_{1,1}(z)|\alpha|+z\widehat{\lambda}_{1,2}(z).$$ Morever,
   since $\widehat{P}_n(0) = 0$, we can write $\widehat{P}_n(z)=z {P}_n^*(z)$, where $\deg P_n^* \leq n-1$, and \eqref{JAPI*} reads
\begin{align}\label{WZ}
   z\left( {{P}^*_n}(z)- \widetilde{P}_n(z)\widehat{\lambda}_{1,1}(z)+Q_n(z)\widehat{\lambda}_{1,2}(z)\right) =\mathcal{O}(1/z^{n+1}).
   \end{align}

    From (\ref{W}) and (\ref{WZ})
     one sees that equations \eqref{JAPI*}, \eqref{JAPIIa*} and \eqref{JAPIIb*} are equivalent to
	\begin{align}
	 P^*_n(z)-\widetilde P_n(z)\widehat{\lambda}_{1,1}(z)+Q_n(z)\widehat{\lambda}_{1,2}(z)&=\mathcal{O}(1/z^{n+2}) \label{JJAPI*},\\
	\widetilde P_n(z)-Q_n(z)\widehat{\lambda}_{2,2}(z)&=\mathcal{O}(1), \label{JJAPIIa*}\\
	 P^*_n(z)-Q_n(z)\widehat{\lambda}_{2,1}(z)&=\mathcal{O}(1/z), \label{JJAPIIb*}
	\end{align}
	where $\deg Q_n\leq n, \deg \widetilde P_n\leq n-1,  P^*_n \leq n-1$. As well, with this new notation in place, \eqref{eq:W1W2}  takes the form
	\begin{equation} \label{eq:Plucker}
	\widehat{\lambda}_{2,1}(z)-\widehat{\lambda}_{2,2}(z)\widehat{\lambda}_{1,1}(z)+\widehat{\lambda}_{1,2}(z)=0.
	\end{equation}
	As before, \eqref{JJAPIIb*} follows from \eqref{eq:Plucker}, \eqref{JJAPI*}, and \eqref{JJAPIIa*}.  We note
	that \eqref{JJAPI*} is a type I Hermite-Pad\'e approximation
	to the direct Nikishin system $\mathcal{N}(\lambda_1, \lambda_2)$,
	while \eqref{JJAPIIa*} and \eqref{JJAPIIb*} is a type II Hermite-Pad\'e
	approximation to the reverse Nikishin system $\mathcal{N}(\lambda_2, \lambda_1)$.
	
\medskip
In the present paper, we  considered a slightly more symmetric
version of the approximation problem \eqref{JJAPI*}, \eqref{JJAPIIa*}.  After dividing equation (\ref{JJAPI*}) by $\widehat{\lambda}_{1,1}$ we get
\begin{equation}
\left( \frac{ P^*_n}{\widehat{\lambda}_{1,1}}-\widetilde P_n+Q_n\frac{\widehat{\lambda}_{1,2}}{\widehat{\lambda}_{1,1}}\right)(z) =\mathcal{O}(1/z^{n+1}).
\end{equation}
Using  formulas (\ref{s22}) and (\ref{4.4}), it follows that
\begin{equation}\label{tra3}
 \left( \frac{z P^*_n + Q_n|\lambda_{1,2}|}{|\lambda_{1,1}|} + b P^*_n  - \widetilde{P}_n +   P^*_n\widehat{\tau}_{1,1}-  \langle \tau_{1,1},\lambda_{2,1} \widehat{\rangle}Q_n\right) (z)=\mathcal{O}(1/z^{n+1}),
 \end{equation}
 where $b \in \mathbb{R}$.

\medskip
Set $\sigma_1:=\tau_{1,1}$ and $\sigma_2:= \lambda_{2,1}$. Take
\[a_{n,2}:= - Q_n, \qquad   a_{n,1}= -  P^*_n, \qquad a_{n,0}= \frac{z P^*_n + Q_n|\lambda_{1,2}|}{|\lambda_{1,1}|} + b P^*_n  - \widetilde{P}_n.\]
From (\ref{JJAPIIb*}) and (\ref{eq:Plucker}) it is easy to check that  $\deg (zP^*_n + Q_{n}|\lambda_{1,2}|) \leq n-1$; consequently $\deg a_{n,0}\le n-1$.
Finally,  we can write (\ref{JJAPI*}) and (\ref{JJAPIIb*}) as
	\begin{align*}
	\left(a_{n,0}-a_{n,1}\widehat{s}_{1,1}+a_{n,2}\widehat{s}_{1,2}\right)(z)=\mathcal{O}(1/z^{n+1}), \\
	\left(a_{n,1}-a_{n,2}\widehat{s}_{2,2}\right)(z)=\mathcal{O}(1/z),
\end{align*}
with $\deg a_{n,0}\leq n-1, \deg a_{n,1} \leq n-1,$ and $\deg a_{n,2} \leq n,$ which are the same interpolation conditions of the Hermite-Pad\'e approximation problem  \eqref{JLS1}-\eqref{JLS2}.


\begin{thebibliography}{99}
	


\bibitem{2} R. Beals, D.H. Sattinger, and J. Szmigielski. Multi-peakons and a theorem of Stieltjes. Inverse
Problems {\bf 15} (1999), no. 1, L1-L4.

\bibitem{3} R. Beals, D.H. Sattinger, and J. Szmigielski. Multi-peakons and the classical moment problem. Advances in Math. {\bf 154} (2000), 229-257.


\bibitem{Bertola-Bothner}
M.~Bertola, T.~Bothner. Universality conjecture and results for a model of several couple positive-definite matrices.
Commun.  Math. Phys {\bf 337} (2015), 1077-1141.

\bibitem{Bertola:CBOPs}
M.~Bertola, M.~Gekhtman, and J.~Szmigielski. Cauchy biorthogonal polynomials.
 J. Approx. Theory {\bf 162} (2010), 832-867.


\bibitem{Ble}  P.M. Bleher and A.B.J. Kuijlaars. Random matrices with external source and multiple orthogonal polynomials. Internat. Math. Research Notices {\bf 3} (2004), 109-129.

\bibitem{Bus} J. Bustamante and G. L\'opez Lagomasino.   Hermite--Pad\'e approximation for Nikishin systems of analytic functions. Russian Acad. Sci. Sb. Math. {\bf 77} (1994), 367-384.

\bibitem{4} R. Camassa and D.D. Holm. An integrable shallow water equation with peaked solitons. Phys.
Rev. Lett. {\bf 71} (1993),   1661-1664.


\bibitem{Car} T. Carleman. Les fonctions quasi-analytiques. Gauthier Villars, Paris, 1926.

\bibitem{Daems}  E. Daems and A.B.J. Kuijlaars. Multiple orthogonal polynomials of mixed type and non intersecting brownian
motions. J. Approx. Theory {\bf 146} (2007), 91-114.

\bibitem{5} A. Degasperis, D.D. Holm, and A.N.W. Hone. A new integrable equation with peakon solutions.
Theoret. and Math. Phys. {\bf 133} (2002),   1463-1474.

\bibitem{6} A. Degasperis, D.D. Holm, and A.N.W. Hone. Integrable and non-integrable equations with peakons. Nonlinear Physics: Theory and
Experiment, II (Gallipoli, 2002) (M. J. Ablowitz, M. Boiti, F. Pempinelli, and B. Prinari, eds.),
World Scientific Publishing, New Jersey, (2003),  37-43.

\bibitem{7} A. Degasperis and M. Procesi. Asymptotic integrability, Symmetry and Perturbation Theory.
(Rome, 1998) (A. Degasperis and G. Gaeta, Eds.), World Scientific Publishing, New Jersey, (1999), 23-37.

\bibitem{LF3} U. Fidalgo and G. L\'opez Lagomasino. General results on the convergence of  multipoint Hermite-Pad\'e approximants of Nikishin systems. Constr. Approx. {\bf 25} (2007), 89-107.

\bibitem{FL3}
U. Fidalgo and G. L\'opez Lagomasino.   Nikishin systems are perfect. Constructive Approx. {\bf 34} (2011), 297-356.

\bibitem{FL4}
U. Fidalgo and G. L\'opez Lagomasino.   Nikishin systems are perfect. Case of unbounded and touching supports. J. of Approx. Theory {\bf 163} (2011), 779-811.

\bibitem{Gon} A.A. Gonchar. On the convergence of generalized Pad\'e approximants of meromorphic functions.
Sb. Math. {\bf 27} (1975), 503-514.

\bibitem{GRS}
A.A. Gonchar, E.A. Rakhmanov, and V.N. Sorokin. On Hermite-Pad\'e
approximants for systems of functions of Markov type. Sbornik: Mathematics {\bf 188} (1997), 671-696.

\bibitem{Her}  Ch. Hermite.  Sur la fonction exponentielle. C. R. Acad. Sci. Paris {\bf 77} (1873), 18-24, 74-79, 226-233, 285-293; reprinted in his Oeuvres, Tome III, Gauthier-Villars, Paris, 1912, 150-181.

\bibitem{krein-string}
M.G. Krein. On inverse problems for a nonhomogeneous cord.
 Doklady Akad. Nauk SSSR (N.S.), {\bf 82} (1952), 669-672.

\bibitem{KN} M.G. Krein and A.A. Nudel'man. The Markov Moment Problem and Extremal Problems. Transl. of Math. Monog. Vol. 50, Amer. Math. Soc., Providence, R.I., 1977.

\bibitem{Kuij} A.B.J. Kuijlaars. Multiple orthogonal polynomial ensembles. Contemporary Mathematics, Vol. 507, Amer. Math. Soc., Providence, R.I. (2010), 155-176.

\bibitem{Kuij2} A.B.J. Kuijlaars. Multiple orthogonal polynomials in random matrix theory. Proc. Internat. Congress
Math., Vol. III, 1417-1432. Hydebarad, India, 2010.

\bibitem{lago} G. L\'opez Lagomasino. Conditions for the convergence of multipoint Pad\'{e} approximants for functions of Stieltjes type. Math. USSR Sb. {\bf 35} (1979), 363-376.

\bibitem{LS} G. L\'opez Lagomasino and S. Medina Peralta.  On the convergence of type I Hermite-Pad\'e approximants.  Advances in Math. {\bf 273} (2015), 124-148.


\bibitem{jacek} H. Lundmark and J. Szmigielski. Degasperis-Procesi peakons and the discrete cubic string. IMRP International Mathematics Research Papers,
2005, No. 2.


\bibitem{Mar} A.A. Markov. Deux d\'emonstrations de la convergence de certains fractions continues. Acta Math. {\bf 19} (1895), 93-104.

\bibitem{Nik} E.M. Nikishin. On simultaneous Pad\'e approximants. Math. USSR Sb. {\bf 41} (1982), 409-425.

\bibitem{Stahl}  H. Stahl. Simultaneous rational approximants. Proceedings of Computaional Mathematics and Function Theory (CMFT'94), World Scientific Publishing Co. R.M Ali, St. Rusheweyh, and E.B. Saff Eds., (1995), 325-349.

\bibitem{Sti} T.J. Stieltjes. Recherches sur les fractions continues. Ann. Fac. Sci. Univ. Toulouse {\bf 8} (1894) J1-J122, {\bf 9} (1895), A1-A47, reprinted in his Oeuvres Compl\`{e}tes, Tome 2, Noordhoff, 1918, 402-566.


\bibitem{walter} W. Van Assche.   Multiple orthogonal polynomials, irrationality and transcendence. Contemporary Mathematics, Vol. 236, Amer. Math. Soc., Providence, R.I. (1999), 325-342.

\end{thebibliography}
\end{document}